\documentclass[11pt]{article}
\usepackage{subfigure}
\usepackage{authblk}
\usepackage{hyperref}

\usepackage{amssymb,amsthm,amsmath,mathrsfs,fullpage,multirow}
\usepackage{tikz}
\usetikzlibrary{matrix}

\makeatletter
\def\keywords#1{%
  \let\@makefnmark\relax 
  \long\def\@makefntext##1{%
    \noindent\textbf{Keywords:} ##1\par}%
  \footnotetext{#1}%
}
\makeatother

\renewcommand\S{\mathcal S}

\newtheorem{theorem}{Theorem}[section]
\newtheorem{lemma}[theorem]{Lemma}

\theoremstyle{definition}

\newtheorem{remark}[theorem]{Remark}
\newtheorem{example}[theorem]{Example}

\newcommand{\thistheoremname}{}
\newtheorem*{genericthm*}{\thistheoremname}
\newenvironment{namedthm*}[1]
  {\renewcommand{\thistheoremname}{#1}%
   \begin{genericthm*}}
  {\end{genericthm*}}

\usepackage{enumerate}

\title{Pattern avoidance in non-crossing and non-nesting permutations}
\author[1]{Kassie Archer}
\author[1]{Robert P. Laudone}
\affil[1]{{\small Department of Mathematics, United States Naval Academy, Annapolis, MD, 21402}}
\affil[ ]{{\small Email: \{karcher, laudone\}@usna.edu }}

\date{}

\begin{document}

\maketitle

\begin{abstract}
    Non-crossing and non-nesting permutations are variations of the well-known Stirling permutations. A permutation $\pi$ on $\{1,1,2,2,\ldots, n,n\}$ is called non-crossing if it avoids the crossing patterns $\{1212,2121\}$ and is called non-nesting if it avoids the nesting patterns $\{1221,2112\}.$ Pattern avoidance in these permutations has been considered in recent years, but it has remained open to enumerate the non-crossing and non-nesting permutations that avoid a single pattern of length 3. In this paper, we provide generating functions for those non-crossing and non-nesting permutations that avoid the pattern 231 (and, by symmetry, the patterns 132, 213, or 312). 
\end{abstract}

\keywords{Pattern avoidance; Stirling permutations; Generating functions.\\
\textbf{2020 Mathematics Subject Classification:} 05A05, 05A15. }

\section{Introduction and Background}

Stirling permutations were defined by Gessel and Stanley in \cite{GS} to be permutations $\pi$ of the multiset $\{1,1,2,2,\ldots, n,n\}$ that avoid the pattern 212. In other words, for any such permutation $\pi$ and indices $i<j<k,$ if $\pi_i=\pi_k,$ then we must have $\pi_j>\pi_i$. In \cite{GS}, they find that the number of such permutations is $(2n-1)!!$ and that their descent polynomials are closely related to the well-known Stirling polynomials (hence their name). Results regarding pattern-avoiding Stirling permutations can be found in \cite{CMM,KP,RW}, among others.  

Stirling permutations are naturally in bijection with increasing ordered rooted trees (see for example \cite{J08}). The authors in \cite{AGPS} removed the increasing condition and asked how this would affect the corresponding permutations. They found that ordered rooted labeled trees are in bijection with a variation of Stirling permutations called non-crossing permutations (also called quasi-Stirling permutations), which avoid the ``crossing'' patterns 1212 and 2121. 
Additionally, non-nesting permutations (also called canon permutations), introduced in \cite{E23} as a variation of non-crossing permutations, are those that avoid the ``nesting'' patterns 1221 and 2112. Both the non-crossing and non-nesting permutations on $\{1,1,2,2,\ldots, n,n\}$ are enumerated by $n!C_n$, where $C_n = \frac{1}{n+1}\binom{2n}{n}$ is the $n$-th Catalan number. 

This fact is easily explained by a bijection between these permutations and certain labeled matchings. Indeed, such a permutation $\pi$ of $\{1,1,2,2,\ldots,n,n\}$ can be viewed as a labeled matching of $[2n]$, by placing an arc with label $\ell$ between $i$ and $j$ if $\pi_i = \pi_j = \ell$.  From this viewpoint, 
the labeled matchings associated to the non-crossing and non-nesting permutations respectively avoid the matchings pictured below:

\begin{center}
    \begin{tikzpicture}
       \draw[line width=.5mm](0,0) to[bend left=45] (2,0);
       \draw[line width=.5mm](1,0) to[bend left=45] (3,0);
               \draw(-0.5,0) -- ++ (4,0);
    \draw[circle,fill] (0,0) circle[radius=1mm]node[below]{$a$};
    \draw[circle,fill] (1,0) circle[radius=1mm]node[below]{$b$};
    \draw[circle,fill] (2,0) circle[radius=1mm]node[below]{$a$};
    \draw[circle,fill] (3,0) circle[radius=1mm]node[below]{$b$};
    \end{tikzpicture} \qquad \qquad
 \begin{tikzpicture}
       \draw[line width=.5mm](0,0) to[bend left=45] (3,0);
       \draw[line width=.5mm](1,0) to[bend left=45] (2,0);
               \draw(-0.5,0) -- ++ (4,0);
    \draw[circle,fill] (0,0) circle[radius=1mm]node[below]{$a$};
    \draw[circle,fill] (1,0) circle[radius=1mm]node[below]{$b$};
    \draw[circle,fill] (2,0) circle[radius=1mm]node[below]{$b$};
    \draw[circle,fill] (3,0) circle[radius=1mm]node[below]{$a$};
    \end{tikzpicture}
\end{center}
Thus, the non-nesting permutations are in bijection with labeled non-nesting matchings while the non-crossing permutations are in bijection with labeled non-crossing matchings. It is well-known \cite{S15} that the number of non-crossing matchings and the number of non-nesting matchings are both equal to $C_n$, and one can assign labels to the matchings in $n!$ ways.

Pattern avoidance in these patterns appears to have some interesting applications. In the non-nesting case, Elizalde and Luo \cite{EL25} explain that there is a relationship to deformations of the braid arrangement (see \cite{B18}). In \cite{E23,E25}, the descent distribution (occurrences of the consecutive 21 pattern) in non-nesting permutations was shown to have an interesting relationship to the Naryana numbers and were used to prove a conjecture about lattice paths. 
In \cite{E21}, Elizalde showed the descent distribution in non-crossing permutations can be expressed as a compositional inverse of the generating function of Eulerian polynomials. In \cite{FL22}, the authors extended the result in \cite{E21} to prove a partial $\gamma$-positivity conjecture and find a bijective proof of an identity originally due to Yan and Zhu \cite{YZ21}.  In the non-crossing case, as mentioned above, these permutations are closely related with labeled rooted ordered trees. One can more generally relate pattern avoidance in the permutation with pattern avoidance in the tree (for example, see Section~\ref{sec:open}).

In \cite{AGPS}, Archer et.~al.~enumerate non-crossing permutations that avoid any set of at least two elements from $\mathcal{S}_3$. In \cite{EL25}, Elizalde and Luo prove analogous results in the non-nesting case. In both papers, the authors pose the question of determining the number of non-crossing (\cite[Section 5]{AGPS}) and non-nesting (\cite[Problem 1]{EL25}) permutations avoiding a single pattern in $\mathcal{S}_3$. In this paper, we resolve one of the two cases of these open questions for both non-crossing and non-nesting permutations. 

We denote by $\mathcal{P}_n(\sigma)$ and $\bar{\mathcal{P}}_n(\sigma)$ the set of non-nesting and non-crossing permutations, respectively, which avoid $\sigma \in \mathcal{S}_3$. We let $p_n(\sigma) = |\mathcal{P}_n(\sigma)|$ and $\bar{p}_n(\sigma) = |\bar{\mathcal{P}}_n(\sigma)|$ and we denote their respective generating functions as 
\[P(x) = \sum_{n\geq0} p_n(231) x^n \quad \text{ and } \quad \bar{P}(x) = \sum_{n\geq0} \bar{p}_n(231) x^n.
\]
In our two main results, we provide generating functions for non-nesting and non-crossing permutations that avoid the pattern $231$ (and, by symmetry, the patterns $132$, $213$ or $312$). 

In Section~\ref{sec:nonnesting-231}, we address the non-nesting case, and in Section~\ref{sec:non-crossing-231}, we address the non-crossing case, proving the following theorems.
\begin{namedthm*}{Theorem~\ref{thm: ImplicitP}}
The generating function $P(x)$ 
    is given implicitly by the equation:
    \[
    x^3P(x)^3 - (x^3+3x^2+x)P(x)^2+(2x^2-x+1)P(x) + x-1 = 0.
    \]
\end{namedthm*}
\begin{namedthm*}{Theorem~\ref{theorem: non crossing 231}}
The generating function $\bar{P}(x)$ is given implicitly by the equation:
\[
x^2\bar{P}(x)^4 - (x^2+x)\bar{P}(x)^3-x\bar{P}(x)^2+(x+1)\bar{P}(x)-1=0.
\]
\end{namedthm*}
In Section~\ref{sec: 122} we enumerate non-crossing permutations avoiding $122$, as well as $122$ and $\sigma$ for $\sigma \in \mathcal{S}_3$. Similar results for non-nesting permutations appear in \cite{EL25}. We conclude with some open questions, the primary one among them being the question of enumerating non-nesting and non-crossing permutations avoiding $321$.

\section{Non-Nesting Permutations avoiding 231}\label{sec:nonnesting-231}

Let $p_n(231) = |\mathcal{P}_n(231)|$ denote the number of non-nesting permutations that avoid $231$, let $q_n(231)$ be the number that begin with $\pi_1 = 1$ and let $r_n(231)$ denote the number that end with $\pi_{2n} = n$. We note that this means $q_0(231) = r_0(231) = 0$.  We begin with some necessary lemmas concerning the positions of $n$ in the permutation, which corresponds to the $n$-arc in the associated matching.

\begin{lemma} \label{lem: nbetween}
    Given $\pi \in \mathcal{P}_n(231)$, if $\pi_i = n$ and $\pi_{j} = n$ with $i < j$, there is at most one arc terminating between $i$ and $j$, and at most one arc originating between $i$ and $j$.
\end{lemma}

\begin{proof}
    First, note that there cannot be an arc that both originates and terminates between $i$ and $j$ because $\pi$ is non-nesting. Suppose there are two arcs that terminate between the two occurrences of $n$. If so, there are two elements $a,b$ with $\pi_{k_1} = \pi_{k_2} = a$ and $\pi_{\ell_1} = \pi_{\ell_2} = b$ where $k_1 < \ell_1 <i< k_2 < \ell_2<j$ since $\pi$ is non-nesting. If $a < b$ then $\pi_{\ell_1} \pi_i \pi_{k_2} = bna$ is a $231$ pattern, and if $a > b$ then $\pi_{k_1} \pi_{i} \pi_{\ell_2} = anb$ is a $231$ pattern.

    A similar argument shows that there is at most one arc originating between the two occurrences of $n$. If not, $\pi$ contains $abnab$ where this $n$ is the second occurrence of $n$, which again produces a $231$.
\end{proof}

\begin{lemma} \label{lem: maxminBetween}
    Given $\pi \in \mathcal{P}_n(231)$, with $\pi_i = n$ and $\pi_j = n$, $i < j$,
    \begin{itemize}
        \item if $\pi_{k_1} = \pi_{k_2} = a$ with $k _1 < i < k_2 < j$ then $a$ is the maximal element to the left of position $i$, and
        \item if $\pi_{k_1}= \pi_{k_2} = a$ with $i < k_1 < j < k_2$, then $a$ is the minimal element to the right of position $j$.
    \end{itemize}
\end{lemma}

\begin{proof}
    There are at least two arcs that originate before $i$, otherwise the statement is trivially true. Now, suppose for sake of contradiction that there are at least two arcs that originate before $i$ and the first bullet point is not true. If $a$ is not maximal, there is some $b$ appearing before position $i$ that is larger than $a$. In this case, $\pi$ contains $bna$ which is a $231$ pattern.

    Similarly, if an element originates between the two occurrences of $n$, but is not minimal, there is some element $b$ to the right of $j$ with $a > b$ in which case $anb$ is a $231$ pattern.
\end{proof}

We are now ready to prove our functional equation relating the generating functions $P(x), Q(x)$ and $R(x)$, where 
 \[P(x) = \sum_{n=0}^\infty p_n(231) x^n, \; Q(x) = \sum_{n=1}^\infty q_n(231) x^n, \;\text{ and } \; R(x) = \sum_{n=1}^\infty r_n(231) x^n.\]

\begin{theorem} \label{thm: non nesting 231}
  The generating functions $P(x), Q(x),$ and $R(x)$ satisfy the functional equation:
    \begin{equation} \label{eq: PRFRecursion}
    P(x) = 2xR(x)Q(x) + xP(x)Q(x) + xR(x) P(x) + xP(x)^2 + 1.
    \end{equation}
\end{theorem}

\begin{proof}
We will prove that for $n \geq 3$, 
    \begin{align*}
    p_n(231) &= 2\sum_{k=1}^{n-2} r_k(231) q_{n-k-1}(231) + \sum_{k=0}^{n-2} p_k(231) q_{n-k-1}(231) \\
    &+ \sum_{k=1}^{n-1} r_k(231) p_{n-k-1}(231)+ \sum_{k=0}^{n-1} p_k(231) p_{n-k-1}(231)
    \end{align*}
    with $p_0(231) = 1$, $p_1(231) = 1$, and $p_2(231) = 4$. This will immediately imply the desired functional equation. 
    
    Given a permutation $\pi \in \mathcal{P}_n(231)$, each summand comes from considering the two possible positions of $n$. There are four options determined by whether an arc originates or terminates inside the $n$-arc, since by Lemma \ref{lem: nbetween} there is at most one arc that terminates and at most one that originates inside the $n$-arc.

First consider the case where there is both a terminal and originating arc inside the $n$-arc. So that $\pi = a_1\cdots a_{2k-1} n a_{2k} b_{1}nb_2\cdots b_{2(n-k-1)}$ or $\pi = a_1\cdots a_{2k-1} n b_{1}a_{2k}nb_2\cdots b_{2(n-k-1)}$, where $a_{2k}$ is a terminal entry and $b_1$ is an initial entry. Throughout this proof, the $a_1\cdots a_{2k}$ are all necessarily less than $b_1\cdots b_{2(n-k-1)}$ so that $\pi$ avoids $231$. 



We must have that $a_{2k}$ is maximal among all the $a_i$ by Lemma \ref{lem: maxminBetween} since it terminates inside the $n$-arc. Similarly, $b_1$ must be minimal among all the $b_i$. As a result, there are $r_k$ choices for $a_1\cdots a_{2k}$ and $q_{n-k-1}$ choices for $b_1\cdots b_{2(n-k-1)}$. The two arcs under the $n$-arc could also cross meaning $a_{2k}$ appears after $b_1$, or $a_{2k}$ could appear before $b_1$, leaving us with $2\sum_{k=1}^{n-2} r_k(231) q_{n-k-1}(231)$ such $231$ avoiding non-nested permutations.

We next consider the case $\pi = a_1\cdots a_{2k} n b_1 n b_2 \cdots b_{2(n-k-1)}$, where there is only an arc originating between the $n$-arc, but not one terminating there. In this case $a_1\cdots a_{2k}$ can be any $231$ avoiding non-crossing permutation, including the empty permutation, but $b_1\cdots b_{2(n-k-1)}$ must again have $b_1$ minimal among the $b_i$'s by Lemma \ref{lem: maxminBetween}. We also must have $n-k-1 \geq 1$ because we must have an entry, $b_1$, between the $n$ entries. This leaves us with $\sum_{k=0}^{n-2} p_k(231) q_{n-k-1}(231)$ such permutations.

Similarly, in the case that $\pi = a_1 \cdots a_{2k-1} n a_{2k} n b_1 \cdots b_{2(n-k-1)}$, $b_1\cdots b_{2(n-k-1)}$ can be any $231$ avoiding non-crossing permutation, including the empty permutation, but $a_{2k}$ must be maximal among the $a_i$'s and $k \geq 1$. This leaves us with $\sum_{k=1}^{n-1} r_k(231) p_{n-k-1}(231)$ such permutations.

\sloppy Finally, consider the case where $\pi = a_1 \cdots a_{2k-1} a_{2k} n n b_1 \cdots b_{2(n-k-1)}$;  there are no arcs originating or terminating between the $n$-arc. Here $a_1\cdots a_{2k}$ and $b_1 \cdots b_{2(n-k-1)}$ can be any $231$ avoiding non-crossing permutations, including the empty permutation, leaving us with $\sum_{k=0}^{n-1} p_k(231) p_{n-k-1}(231)$ options.


This exhausts all possibilities, summing yields the stated recursion which gives the functional equation among the corresponding generating functions.
\end{proof}

\begin{example}
    Consider the example $\pi = 121632653454$ with the following matching diagram:
    
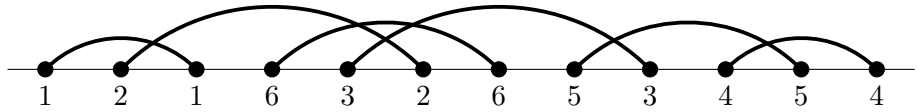
\begin{center}
    \begin{tikzpicture}
       \draw[line width=.5mm](0,0) to[bend left=45] (2,0);
       \draw[line width=.5mm](1,0) to[bend left=45] (5,0);
       \draw[line width=.5mm](3,0) to[bend left=45] (6,0);
       \draw[line width=.5mm](4,0) to[bend left=45] (8,0);
       \draw[line width=.5mm](7,0) to[bend left=45] (10,0);
       \draw[line width=.5mm](9,0) to[bend left=45] (11,0);
               \draw(-0.5,0) -- ++ (12,0);
    \draw[circle,fill] (0,0) circle[radius=1mm]node[below]{$1$};
    \draw[circle,fill] (1,0) circle[radius=1mm]node[below]{$2$};
    \draw[circle,fill] (2,0) circle[radius=1mm]node[below]{$1$};
    \draw[circle,fill] (3,0) circle[radius=1mm]node[below]{$6$};
    \draw[circle,fill] (4,0) circle[radius=1mm]node[below]{$3$};
    \draw[circle,fill] (5,0) circle[radius=1mm]node[below]{$2$};
    \draw[circle,fill] (6,0) circle[radius=1mm]node[below]{$6$};
    \draw[circle,fill] (7,0) circle[radius=1mm]node[below]{$5$};
    \draw[circle,fill] (8,0) circle[radius=1mm]node[below]{$3$};
    \draw[circle,fill] (9,0) circle[radius=1mm]node[below]{$4$};
    \draw[circle,fill] (10,0) circle[radius=1mm]node[below]{$5$};
    \draw[circle,fill] (11,0) circle[radius=1mm]node[below]{$4$};
    \end{tikzpicture}
\end{center}

    This falls into the case where there are entries to the left and right of the largest entry, $6$. There are also arcs originating and terminating between the two occurrences of $6$. So this permutation decomposes into a permutation with two arcs ending in a $2$: $1212$, indicated in red, and a permutation with three arcs beginning with a $3$: $353454$, indicated in blue.
    
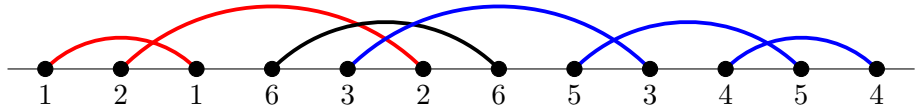
\begin{center}
    \begin{tikzpicture}
       \draw[line width=.5mm, color = red](0,0) to[bend left=45] (2,0);
       \draw[line width=.5mm, color = red](1,0) to[bend left=45] (5,0);
       \draw[line width=.5mm](3,0) to[bend left=45] (6,0);
       \draw[line width=.5mm, color = blue](4,0) to[bend left=45] (8,0);
       \draw[line width=.5mm, color = blue](7,0) to[bend left=45] (10,0);
       \draw[line width=.5mm, color = blue](9,0) to[bend left=45] (11,0);
               \draw(-0.5,0) -- ++ (12,0);
    \draw[circle,fill] (0,0) circle[radius=1mm]node[below]{$1$};
    \draw[circle,fill] (1,0) circle[radius=1mm]node[below]{$2$};
    \draw[circle,fill] (2,0) circle[radius=1mm]node[below]{$1$};
    \draw[circle,fill] (3,0) circle[radius=1mm]node[below]{$6$};
    \draw[circle,fill] (4,0) circle[radius=1mm]node[below]{$3$};
    \draw[circle,fill] (5,0) circle[radius=1mm]node[below]{$2$};
    \draw[circle,fill] (6,0) circle[radius=1mm]node[below]{$6$};
    \draw[circle,fill] (7,0) circle[radius=1mm]node[below]{$5$};
    \draw[circle,fill] (8,0) circle[radius=1mm]node[below]{$3$};
    \draw[circle,fill] (9,0) circle[radius=1mm]node[below]{$4$};
    \draw[circle,fill] (10,0) circle[radius=1mm]node[below]{$5$};
    \draw[circle,fill] (11,0) circle[radius=1mm]node[below]{$4$};
    \end{tikzpicture}
\end{center}

This is one of the permutations enumerated by $r_2q_3$. Notice, it would not create a $231$ to swap the entries between the two occurrences of $6$, as seen below.
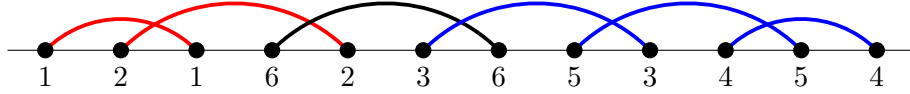
\begin{center}
    \begin{tikzpicture}
       \draw[line width=.5mm, color = red](0,0) to[bend left=45] (2,0);
       \draw[line width=.5mm, color = red](1,0) to[bend left=45] (4,0);
       \draw[line width=.5mm](3,0) to[bend left=45] (6,0);
       \draw[line width=.5mm, color = blue](5,0) to[bend left=45] (8,0);
       \draw[line width=.5mm, color = blue](7,0) to[bend left=45] (10,0);
       \draw[line width=.5mm, color = blue](9,0) to[bend left=45] (11,0);
               \draw(-0.5,0) -- ++ (12,0);
    \draw[circle,fill] (0,0) circle[radius=1mm]node[below]{$1$};
    \draw[circle,fill] (1,0) circle[radius=1mm]node[below]{$2$};
    \draw[circle,fill] (2,0) circle[radius=1mm]node[below]{$1$};
    \draw[circle,fill] (3,0) circle[radius=1mm]node[below]{$6$};
    \draw[circle,fill] (4,0) circle[radius=1mm]node[below]{$2$};
    \draw[circle,fill] (5,0) circle[radius=1mm]node[below]{$3$};
    \draw[circle,fill] (6,0) circle[radius=1mm]node[below]{$6$};
    \draw[circle,fill] (7,0) circle[radius=1mm]node[below]{$5$};
    \draw[circle,fill] (8,0) circle[radius=1mm]node[below]{$3$};
    \draw[circle,fill] (9,0) circle[radius=1mm]node[below]{$4$};
    \draw[circle,fill] (10,0) circle[radius=1mm]node[below]{$5$};
    \draw[circle,fill] (11,0) circle[radius=1mm]node[below]{$4$};
    \end{tikzpicture}
\end{center}

This also yields a valid 231-avoiding non-nesting permutation, which is why we must  have a factor of 2 in the first summand of Theorem \ref{thm: non nesting 231}.
\end{example}

We now need a better understanding of $R(x)$ and $Q(x)$ to get our implicit equation for $P(x)$.

\begin{lemma} \label{lem: RinTermsOfP}
    For $n \geq 2$,
    $
    r_n = p_{n-1} + r_{n-1}
    $
    and $r_1 = 1$. This implies
    \[
    R(x) = \frac{xP(x)}{1-x}.
    \]
    
\end{lemma}

\begin{proof}
    Given $\pi \in \mathcal{P}_n(231)$  with $\pi_{2n}=n$, by Lemma \ref{lem: nbetween} there is at most one arc that terminates between the two occurrences of $n$, and no arc can originate because $\pi_{2n} = n$ and $\pi$ is non-nesting.

    This means that either $\pi_{2n-1} = n$ or $\pi_{2n-2} = n$. In the first case, we clearly have $p_{n-1}$ options because the remaining permutation is $231$ avoiding and non-nesting. In the second case, we have $r_{n-1}$ options by Lemma \ref{lem: maxminBetween} because $\pi_{2n-1}$ must be the maximal element to the left of position $2n-2$, forcing $\pi_{2n-1} = n-1$, meaning $\pi_1\dots\pi_{2n-3}\pi_{2n-1} \in \mathcal{P}_{n-1}(231)$ with last entry equal to $n-1$.

    The result about the generating functions follows immediately from the recursion since $R(x) = xP(x) + xR(x)$.
\end{proof}




Now, let $r'_n$ denote the $231$-avoiding non-nested permutations with both $\pi_1 = 1$ and $\pi_{2n} = n$ and denote its corresponding generating function by 
\[R'(x) = \sum_{n=1}^\infty r_n' x^n.\]

\begin{lemma} \label{lem: R'inTermsOfF}
    For $n \geq 2$,
    $
    r_n' = q_{n-1} + r_{n-1}'
    $
    and $r_1' = 1$. Thus, we have
    \[
    R'(x) = \frac{x+xQ(x)}{1-x}
    \]
\end{lemma}

\begin{proof}
The argument here is identical to the argument in Lemma \ref{lem: RinTermsOfP}, but now we always force $\pi_1 = 1$. So if $\pi$ is a $231$-avoiding non-nesting permutation with $\pi_1 = 1$ and $\pi_{2n} = n$ we could have $\pi_{2n-1} = n$ or $\pi_{2n-2} = n$. In the first case, we get $q_{n-1}$ permutations because the remaining entries form a 213-avoiding non-nesting permutation that begins with $1$. In the second case, we get $r'_{n-1}$ permutations because the remaining elements form and 231-avoiding non-nesting permutation that begins with $1$ and ends with $n-1$ by Lemma \ref{lem: maxminBetween}. We note that we consider $r_1' = 1$ since $\pi = 11$ both begins and ends with the largest element. This shows $R'(x) = xQ(x) + xR'(x) + x$, we need the additional $x$ term because neither $Q(x)$ and $R'(x)$ have a constant term. The statement about generating functions follows immediately.
\end{proof}

Next, let us focus on those permutations with $\pi_1 = 1$. 

\begin{lemma} \label{lem: FR'P}
    The generating function $Q(x)$ satisfies the following functional equation:
    \[
    Q(x) = 2xR'(x)Q(x) + xQ(x)^2 + xR'(x)P(x) + xQ(x)P(x) + x.
    \]
\end{lemma}

\begin{proof}
    The argument is identical to that  from Theorem \ref{thm: non nesting 231}, replacing all the first terms in each recursion with their corresponding term that begins with $1$.
\end{proof}

We are now ready to prove the main result of this section. 

\begin{theorem} \label{thm: ImplicitP}
    The generating function
    \[
    P(x) = \sum_{n\geq0} p_n(231) x^n
    \]
    is given implicitly by
    \[
    x^3P(x)^3 - (x^3+3x^2+x)P(x)^2+(2x^2-x+1)P(x) + x-1 = 0.
    \]
\end{theorem}

\begin{proof}
Substituting $R'(x) = \frac{x+xQ(x)}{1-x}$ from Lemma \ref{lem: R'inTermsOfF} into the identity in Lemma \ref{lem: FR'P} and solving for $P(x)$, we get: 
 \[
    P(x) = \frac{-x + x^2 + Q(x) - xQ(x) - 2x^2Q(x)- xQ(x)^2-x^2Q(x)^2}{x(x+Q(x))}
    \]
Substituting  $R(x) = \frac{xP(x)}{1-x}$ from Lemma \ref{lem: RinTermsOfP} into \eqref{eq: PRFRecursion} and solving for $Q(x)$, we obtain: 
    \[
    Q(x) = \frac{-1 + x + P(x) - xP(x) - xP(x)^2}{x(1+x)P(x)}
    \]
Finally, one can substitute this equation for $Q(x)$ into the equation for $P(x)$, and
after simplifying, one obtains the desired implicit equation. 
\end{proof}

We can use the functional equation in Theorem \ref{thm: ImplicitP} to compute:
\[
P(x) = 1 + x+4x^2+17x^3+77x^4+367x^5+1815x^6+9233x^7+48014x^8+254123x^9+1364491x^{10}+\cdots
\]
with more coefficients found at \cite[A383770]{OEIS}

\begin{remark}
If you solve for $P(x)$ explicitly, the following radical appears:
\[
\sqrt{-x^8+4x^9-2x^{10}+92x^{11}+47x^{12}-140x^{13}-76x^{14}+16x^{15}-8x^{16}}.
\]
The polynomial under the radical has a minimal root of $0.161809$, giving a growth rate of $\frac{1}{0.161809} \approx 6.1801$ for $p_n(231)$. For example, we note that $p_{250}/p_{249}  \approx 6.143$. The largest growth rate for non-nesting permutations avoiding two patterns in $\mathcal{S}_3$ was $3$ \cite{EL25}.
\end{remark}

\section{Non-crossing permutations avoiding 231} \label{sec:non-crossing-231}

Let $\bar{p}_n=\bar{p}_n(231)$ be the number that avoid 213 and let $\bar{q}_n$ be the number that start with $\pi_1=1$. 

\begin{lemma}\label{lemma: p-bar}
For $n\geq 1$,
\[
\bar{p}_n = \sum_{k=2}^n\sum_{i=1}^{k-1}\bar{p}_{i}\bar{p}_{n-k}\bar{p}_{n-k-1-i} +  \sum_{k=1}^n\bar{p}_{k-1}\bar{q}_{n+1-k} 
\] where $\bar{p}_0=1$.
\end{lemma}

\begin{proof}
Let us consider those permutations that start with $\pi_1=k$ for some $k$. Since $\pi$ avoids 231, all elements less than $k$ must appear in the permutation before all elements greater than $k$. Let us consider two cases. 

First,  suppose $\pi = k A B k C$ where every element of $A$ is less than $k$, which implies that the elements of $B$ and $C$ are exactly those that are greater than $k$. Notice that $A$ can be any non-crossing permutation of semi-length $k-1$ and $kBkC$ is any non-crossing permutation of semi-length $n-k$ that starts with its smallest element. Thus there are a total of $\bar{p}_{k-1}\bar{q}_{n-k}$ non-crossing 231-avoiding permutations of this form for any $1\leq k \leq n$.

Next, suppose $\pi = k A k B C$ where every element of $C$ is greater than $k$, which implies that the elements of $A$ and $B$ are exactly composed of those elements that are less than $k$. Furthermore, suppose $B$ is nonempty so there is no overlap with the previous case. Since $\pi$ avoids 231, it must be that the elements of $B$ are greater than those of $A$. Note also that $A$, $B$, and $C$ are necessarily equivalent to non-crossing permutations themselves. Letting $i$ be the semilength of $B$, then we must have that there are \[\sum_{i=1}^{k-1}\bar{p}_{n-k-i-1}\bar{p}_i\bar{p}_{n-k}\] permutations of this form.
\end{proof}

\begin{lemma}\label{lemma: q-bar}
For $n\geq 1$,
\[
\bar{q}_k = \sum \bar{p}_{x_1-1}\bar{p}_{x_2-1}\cdots\bar{p}_{x_k-1}
\] where the sum is over all compositions $\boldsymbol{x}$ of $n$.
\end{lemma}

\begin{proof}
Suppose that $\pi_1=1$ and $\pi_i=1$. Since $\pi$ avoids 231, then we must have that $\pi_2\geq \pi_3\geq \cdots \geq \pi_{i-1}$. Furthermore, we can write $\pi$ as 
\[
\pi = 1a_1a_1a_2a_2\ldots a_{k-1}a_{k-1} 1 \beta_{k}\beta_{k-1}\ldots \beta_1
\] 
where for each $1\leq j \leq k$, all elements of $\beta_j$ are less than $a_{j-1}$ and greater than $a_{j}$. Note that the $\beta_j$ can be any 231-avoiding non-crossing permutation, and if $a_j=a_{j-1}-1$, then $\beta_j$ is empty. Since taking $x_i=|\beta_i|+1$ gives us a composition $\boldsymbol{x}$ of $n$, and the elements $a_i$ and those that appear in $\beta_i$ are forced by the inequalities present, the result follows. 
\end{proof}

\begin{example}
    Let us demonstrate Lemma~\ref{lemma: q-bar} with an example. Consider the permutation \[\pi = 11773312246655488.\] 
    In this case $a_1=7$ and $a_2=3,$ $\beta_3=22$, $\beta_2=466554$, and $\beta_1 = 88.$ Note that the associated composition is $2+4+2$ since the sizes of $\beta_3,\beta_2,$ and $\beta_1$ are 1, 3, and 1, respectively.
\end{example}

\begin{theorem}\label{theorem: non crossing 231}
The generating function given by \[\bar{P}(x) =\displaystyle\sum_{n\geq 0} \bar{p}_n(231) x^n\] is given implicitly by the equation:
\[
x^2\bar{P}(x)^4 - (x^2+x)\bar{P}(x)^3-x\bar{P}(x)^2+(x+1)\bar{P}(x)-1=0.
\]
\end{theorem}

\begin{proof}
    If we let $\bar{Q}(x) =\displaystyle\sum_{n\geq 0} \bar{q}_n x^n$, then Lemma~\ref{lemma: p-bar} tells us 
    \[
    \bar{P}(x) = x\bar{P}(x)\bar{P}(x)(\bar{P}(x)-1) + \bar{P}(x)\bar{Q}(x) + 1
    \]
    and Lemma~\ref{lemma: q-bar} tells us
    \[
    \bar{Q}(x) = \frac{x\bar{P}(x)}{1-x\bar{P}(x)}.
    \]
    Solving for $\bar{P}(x)$, we obtain the result.
\end{proof}

We can use the functional equation in Theorem \ref{theorem: non crossing 231} to compute
\[
\bar{P}(x) =1+x+4x^2+19x^3+102x^4+590x^5+ 3588x^6 + 22617x^7 + 146460x^8 + 968520x^9 +\cdots.
\]
with more coefficients found at \cite[A383771]{OEIS}.
\begin{remark}
Using Lagrange inversion, we can get the following formula for $\bar{p}_n(231)$: 
\[
\bar{q}_n(213) = \frac{1}{n} \sum_{j=0}^n\sum_{r=0}^{n-j}\sum_{m=0}^{n-1}\sum_{\ell=0}^m \frac{6^{m-\ell}}{2^n}\binom{n}{j}\binom{n-j}{r}\binom{m}{\ell}\binom{j/2}{m}\binom{2r+m+\ell+n}{j+r-m-\ell-1}.
\]
\end{remark}

 If you solve for $\bar{P}(x)$ explicitly, the following radical appears:
\[
\sqrt{-108x^3 + 621 x^4 + 432 x^5 +10206x^6 + 432 x^7 + 621 x^8 - 108 x^9},
\]
which has minimal root $x=.12791$, and $1/.12791 = 7.81774$, which gives us the growth rate for $\bar{p}_n(231)$. For example, $\bar{p}_{300}/\bar{p}_{299} = 7.77875$ and $\bar{p}_{600}/\bar{p}_{599} = 7.79822$. 
Recall that the growth rate for non-nesting permutations that avoid 213 is 6.1801, which implies that it is easier for a non-crossing permutation to avoid 231 than it is for a non-nesting permutation to avoid 231. 

\section{Non-crossing permutations avoiding 122} \label{sec: 122}

Note that in \cite{EL25}, non-nesting permutations avoiding patterns with repeated elements were considered, but this case was omitted for non-crossing permutations in \cite{AGPS}. Here, let us consider those non-crossing permutations that avoid the pattern $122$ (or equivalently, avoiding a single pattern in the set $\{112, 211, 221\}$), as well as those that avoid both 122 and another pattern $\sigma\in\S_3.$ We don't consider the pattern 212 or 121 since this is equivalent to Stirling permutations, whose pattern avoidance has already been well-studied.

\begin{theorem}\label{thm:112}
For $n\geq1$, 
\[
\bar{q}_n(122) = C_n
\] where $C_n$ is the $n$-th Catalan number.
\end{theorem}

\begin{proof}
Start with any noncrossing matching and label left-to-right in decreasing way. This will clearly avoid 122 since any subsequence $baab$ or $bbaa$ must have that $a<b$ if we label according to this rule. If we label it any other way, then we must have a $a>b$ and thus a 122 pattern.
\end{proof}

\begin{theorem}
For $n\geq1$, 
\[
\bar{q}_n(122,\sigma) =\begin{cases} 
C_n & \sigma =132\\
F_{n+1} & \sigma=213 \\
2^{n-1} & \sigma\in\{231,123\} \\
n & \sigma=312
\end{cases}
\]
and $\bar{q}_n(122,321)=0$ for $n\geq 3.$
\end{theorem}

\begin{proof}
Let us first note that by the proof of Theorem~\ref{thm:112}, any non-crossing permutation that avoids 122 must be obtained from labeling a non-crossing matching left-to-right in a decreasing way. In particular, for any $i$, the first occurrence of $i$ must appear before the first occurrence of $i-1$. 

It is clear no such labeling could allow there to be a 132 pattern. Indeed, for any $a<b<c,$ a 132 pattern would be of the form $acb$ and thus $c$ and $b$ would be the second occurrence of that element. But then $cbcb$ would be a subsequence of $\pi$, contradicting that it is non-crossing. It follows that $\bar{q}_n(122,132)=\bar{q}_n(122)=C_n$.

Let us consider those that avoid 213. If there are more than two distinct elements between the two occurrences of $n$, they must appear in nondecreasing order (since $\pi$ avoids 213) and so would form a $122$ pattern. Thus every permutation 
must begin with $nn$ or $n(n-1)(n-1)n$, and so $\bar{q}_n(122,213)=\bar{q}_{n-1}(122,213)+\bar{q}_{n-2}(122,213)$. Since $\bar{q}_1(122,213) =1$ and $\bar{q}_2(122,213) =2$, we have $\bar{q}_n(122,213) = F_n.$

Now let's consider those that avoid 231. 
Note that these permutations must be of the form $nAnB$ where the elements of $A$ are less than the elements of $B$. If both $A$ and $B$ are nonempty, then there is a subsequence $abb$ of $\pi$ that forms a 122 pattern. Therefore either $\pi_2=n$ or $\pi_{2n}=n.$ Since we can recursively obtain all permutations of semilength $n$ from those of semilength $n-1$ by inserting $nn$ at the beginning or by inserting $n$ at the front and $n$ at the end, $\bar{q}_n(122,231)=2\bar{q}_{n-1}(122,231)$. Together with the fact that $\bar{q}_1(122,231) =1$, we have $\bar{q}_n(122,231)=2^{n-1}.$

Next, consider those that avoid 123. Notice we cannot have a nest with three elements, i.e., we cannot have $cbaabc$ as a subsequence for any $a,b,c$ since we would have $c>b>a$ and thus $abc$ would be a 123 pattern. Therefore, our permutation is composed of contiguous segments of the form $i(i-1)(i-1)(i-2)(i-2)\ldots (i-k)(i-k)i$ and so these permutations are in bijection with compositions of $n$. It follows that $\bar{q}_n(122,123)=2^{n-1}.$

Next, if $\pi$ avoids 312, then since $\pi_1=n$, all other elements must appear in decreasing order. Thus $\pi = n(n-1)(n-1)(n-2)(n-2)\ldots(n-k)(n-k)n(n-k-1)(n-k-1)\ldots2211$ for some $0\leq k\leq n-1$; which proves $\bar{q}_n(122,312)=n.$

Finally, any permutation avoiding 122 must eventually contain a 321 pattern and so for $n\geq 3$, we have $\bar{q}_n(122,321)=0.$
\end{proof}

\section{Open Questions}\label{sec:open}

There are many future directions for this research. In particular, it is still open to enumerate non-nesting and non-crossing permutations that avoid the pattern 321. It is clear that both $p_n(321)$ and $\bar{p}_n(321)$ are bounded between $C_n$ and $C_n^2$, but enumerating them exactly seems to challenging. 
It may be reasonable to consider consecutive patterns as well, or to apply the methods in this paper to consider pattern avoidance in permutations on more general multisets that avoid the crossing or nesting patterns.

We also note that as stated in \cite{AGPS}, there is a nice bijection between non-crossing permutations and labeled rooted ordered trees. There has been some interest in pattern avoidance in unordered trees \cite{AA,GP,R24}. Studying pattern avoidance in non-crossing permutations can help address the question of pattern avoidance in ordered trees. In particular, the number of non-crossing permutations that avoid 23132 would be equal to the number of labeled rooted ordered trees that avoid 231.

\subsection*{Disclaimer}

{\emph{The views expressed in this paper are those of the authors and do not reflect the official policy or position of the U.S. Naval Academy, Department of the Navy, the Department of Defense, or the U.S. Government.}}


\end{document}